\documentclass{amsart}

\usepackage{amsmath,amssymb,mathrsfs,color,enumerate}
\usepackage[all]{xy}
\usepackage{calligra}
\DeclareMathAlphabet{\mathcalligra}{T1}{calligra}{m}{n}
\DeclareFontShape{T1}{calligra}{m}{n}{<->s*[2.2]callig15}{}
\usepackage{hyperref}
\hypersetup{
  colorlinks=true,
  linktoc=all,
  linkcolor=blue,
}
\usepackage{fullpage}

\definecolor{note}{rgb}{.80,.20,.20}

\newcommand{\mbb}[1]{\mathbf #1}
\newcommand{\mf}[1]{\mathfrak #1}
\newcommand{\mc}[1]{\mathcal #1}
\newcommand{\ms}[1]{\mathscr #1}

\def\<{\left<}
\def\>{\right>}

\newcommand{\simto}{\stackrel{\sim}{\to}}

\DeclareMathOperator{\isom}{Isom}
\DeclareMathOperator{\aut}{Aut}
\DeclareMathOperator{\Hom}{Hom}
\newcommand{\sHom}{\ms H\!om}
\newcommand{\widebar}{\overline}
\newcommand{\tensor}{\otimes}
\newcommand{\oper}{\operatorname}
\DeclareMathOperator{\Comp}{Comp}

\newcommand{\Ab}{\oper{Ab}}

\theoremstyle{plain}

\newtheorem{thm}{Theorem}[subsection]
\newtheorem{lem}[thm]{Lemma}
\newtheorem{cor}[thm]{Corollary}
\newtheorem{prop}[thm]{Proposition}

\newtheorem*{thm*}{Theorem}
\newtheorem*{question*}{Question}
\newtheorem*{lem*}{Lemma}

\newtheorem*{cor*}{Corollary}
\newtheorem*{prop*}{Proposition}
\newtheorem*{claim*}{Claim}

\theoremstyle{definition}
\newtheorem{rem}[thm]{Remark}
\newtheorem{notation}[thm]{Notation}
\newtheorem{defn}[thm]{Definition}
\newtheorem{example}[thm]{Example}
\newtheorem*{rem*}{Remark}

\newcommand{\ov}[1]{\overline{#1}}
\newcommand{\til}[1]{\widetilde{#1}}

\newcommand{\bP}{\mbb P}
\newcommand{\Z}{\mbb Z}

\newcommand{\m}{\boldsymbol{\mu}}
\newcommand{\send}{\ms E\!nd}

\DeclareMathOperator{\Br}{Br}

\DeclareMathOperator{\Pic}{Pic}
\DeclareMathOperator{\per}{per}
\DeclareMathOperator{\ind}{ind}
\DeclareMathOperator{\Spec}{Spec}
\DeclareMathOperator{\spec}{Spec}

\DeclareMathOperator{\rk}{rank}

\newcommand{\weird}{(\textbf C)}
\newcommand{\wweird}{(\textbf C\(\star\))}

\newcommand{\ul}{\underline}
\newcommand{\op}{\operatorname}

\newcommand{\LF}[1]{\op{\ul{LF}/{#1}}}
\newcommand{\FLF}[1]{\op{\ul{FLF}/{#1}}}

\newcommand{\Shv}[1]{\op{\ul{Shv}/{#1}}}
\newcommand{\Alg}[1]{\op{\ul{Alg}/{#1}}}
\newcommand{\Unitalization}{\operatorname{Un}}

\newcommand{\LFA}[1]{\op{\ul{LFAlg}/{#1}}}
\newcommand{\UMod}[1]{\op{\ul{UQC}/{#1}}}
\newcommand{\Agree}[1]{\op{\ul{Agree}/{#1}}}

\newcommand{\Rep}[2]{{{\op{Rep}}_{#1}^{#2}}}
\newcommand{\Ulr}[2]{{{\op{Ulr}}_{#1}^{#2}}}
\newcommand{\AzRep}[2]{{\op{AzRep}_{#1}^{#2}}}
\newcommand{\PrUlr}[2]{{\op{PrUlr}_{#1}^{#2}}}
\newcommand{\PrSUlr}[2]{{\op{PrSplUlr}_{#1}^{#2}}}
\newcommand{\SUlr}[2]{{\op{SplUlr}_{#1}^{#2}}}

\newcommand{\stRep}[2]{\underline{{\mc Rep}}_{#1}^{#2}}
\newcommand{\stUlr}[2]{\underline{{\mc Ulr}}_{#1}^{#2}}
\newcommand{\stSUlr}[2]{\underline{{\mc Spl\mc Ulr}}_{#1}^{#2}}
\newcommand{\stAzRep}[2]{\underline{{\mc Az\mc Rep}}_{#1}^{#2}}
\newcommand{\stPrUlr}[2]{\underline{{\mc Pr\mc Ulr}}_{#1}^{#2}}
\newcommand{\stPrSUlr}[2]{\underline{{\mc Pr\mc Spl\mc Ulr}}_{#1}^{#2}}

\newcommand{\stPic}[2]{\underline{{\mc Pic}}_{#1}^{#2}}
\newcommand{\shPic}[2]{{{\mc Pic}}_{#1}^{#2}}

\newcommand{\CFun}[1]{\ms C\!\!\ms F_{#1}}
\newcommand{\CAlg}[1]{\ms C_{#1}}

\newcommand{\Free}[1]{\op{F}\!\left<{#1}\right>}
\newcommand{\freakyideal}[1]{\ms I({#1})}

\title{The Clifford algebra of a finite morphism}

\author{Daniel Krashen}
\author{Max Lieblich}

\begin{document}

\begin{abstract}
We develop a general theory of Clifford algebras for finite morphisms of
schemes and describe several applications to the theory of Ulrich bundles
and connections to period-index problems for curves of genus 1.
\end{abstract}

\maketitle

\tableofcontents

\section{Introduction}

The goal of this paper is to develop a general theory of Clifford
algebras for finite morphisms of schemes, with a view toward the theory of Ulrich bundles and period-index theorems for genus $1$ curves.

A construction of Roby \cite{Roby}, defines a Clifford algebra, denoted $C(f)$, associated
to
a homogeneous form $f$ of degree $d$ in $n$ variables on a vector space $V$
(the classical Clifford algebra arising in the case that $d = 2$). The
behavior of this algebra is intimately connected with the geometry of the hypersurface $X$
defined by the equation $x_0 ^d - f(x_1, \ldots, x_n)$ in $\bP^n$. The
classical results leads one to believe that perhaps the Clifford algebra of
the form $f$ is intrinsic to the variety $X$.

As we explain here, the Clifford algebra of a form is really a structure
associated not to a scheme $X$, but to a finite morphism $\phi:X\to Y$, designed to (co)represent a
functor on the category of locally free algebras over the base scheme $S$.
Roughly speaking, the Clifford algebra of $\phi$ is a locally free sheaf
$\ms C$ of (not necessarily commutative) $\ms O_S$-algebras such that maps
from $\ms C$ into any locally free $\ms O_S$-algebra $\ms B$ are the same
as maps from $\phi_\ast\ms O_X$ into $\ms B|_Y$. Taking $\ms B$ to be a
matrix algebra, we see that the representations of such a $\ms C$
parameterize sheaves on $X$ with trivial pushforward to $Y$, yielding a
connection to Ulrich bundles. Making this idea work is somewhat delicate
and requires various hypotheses on $X$ and $Y$ that we describe in the
text.

In the case of the classical Clifford algebra of a form described above, if one
takes $Y$ to be the projective space $\bP^{n-1}$ and the morphism $\phi:X \to
\bP^{n-1}$ to be given by dropping the $x_0$-coordinate, then one obtains a
natural identification $C(\phi) \cong C(f)$. Our construction also
generalizes other constructions of Clifford algebras,
as in \cite{HaileHan,Kuo,ChaKuo} which do not come directly from a
homogeneous form.

\subsection{Structure of paper}
In Section~\ref{existence}, we construct the Clifford algebra of a morphism
satisfying certain conditions (see Definition~\ref{clifford definition},
Theorem~\ref{existence theorem}).  In Section~\ref{ulrich} we study the
representations of the Clifford algebra and their relations to Ulrich
bundles in a general context. In particular, we show that a
natural quotient of the Clifford algebra (the so-called ``reduced Clifford
algebra,'' of Definition~\ref{reduced clifford}) is Azumaya, and its
center is the coordinate ring
for the coarse moduli space of its representations of minimal degree (see
Theorem~\ref{good azumaya love}), generalizing results of Haile and
Kulkarni \cite{Haile:CABC, Kulk:CBi}. Sections~\ref{existence}
and~\ref{ulrich}
work over an arbitrary base scheme $S$. In Section~\ref{curves} we
specialize to the case of a finite morphism from a curve to the projective
line, extending results of \cite{Cos, Kulk:CBi} and relating the Clifford
algebra and its structure to the period-index problem for genus~$1$ curves.
Finally, in Appendix~\ref{examples}, we give more explicit constructions in
the case of morphisms of subvarieties of weighted projective varieties,
relating our construction to the more classical perspective of Clifford
algebras associated to forms.

\subsection*{Acknowledgments}

The first author was partially supported by NSF CAREER Grant DMS-1151252
and NSF Grant DMS-1007462 and the second author was partially supported by
a Sloane Research Fellowship, NSF CAREER Grant DMS-1056129 and NSF Grant
DMS-1021444, during the writing of this paper.

The first author would like to thank Valery Alexeev and Adrian Wadsworth
for useful conversations during the writing of this paper. We are also grateful for
errors pointed out in the earlier version of this paper by Baptiste Calmes as well as by the anonymous referee.

\section{General definition and existence} \label{existence}

\subsection{Notation}
Fix throughout the section a base scheme $S$. For an $S$-scheme $X$, we
will write $\pi_X : X \to S$ for the structure morphism, or simply write
$\pi = \pi_X$ if the context is clear. We write $\Shv{S}$ for the category
of sheaves of sets on $S$. For a sheaf of unital $\ms O_S$-algebras $\ms A$, we
write $\epsilon_{\ms A}: \ms O_S \to \ms A$ for the $\ms O_S$-algebra
structure map. We assume that all algebras (and sheaves of algebras) are
unital and associative. We do, however, allow the possibility of the
$0$-ring, containing a single element in which the elements $0$ and $1$
coincide.

For a quasicoherent sheaf $\mathscr N$ of $\ms O_X$ modules, we write $\ms N^\vee$ to denote the dual sheaf $\ms Hom_{\ms O_X}(\ms N, \ms O_X)$.

Let $\LF{\ms O_S}$ denote the category of locally free quasi-coherent sheaves
of $\ms O_S$ modules (with arbitrary, possibly infinite, local rank) and $\FLF{\ms O_S}$ the
category of finite locally free sheaves (i.e., those with rank an element of $\Gamma(S, \Z)$). Let $\Alg{\ms O_S}$ denote the
category of quasi-coherent sheaves of $\ms O_S$-algebras, and  $\LFA{\ms
O_S}$ denote the full subcategory of those which are locally free as $\ms
O_S$-modules.

\subsection{The Clifford functor and Clifford algebra}
Associated to a morphism $\phi$ of $S$ schemes, we will define a Clifford
functor $\CFun{\phi}$, and under certain assumptions show that it is
representable by an algebra which we refer to as the Clifford algebra
of the morphism $\phi$.

\begin{defn}
Let $\phi: X \to Y$ be a morphism of $S$-schemes. We define the
\textit{Clifford functor of $\phi$} \[\CFun{\phi}:\LFA{\ms O_S}\to\Shv{S}\] by the formula
\[\CFun{\phi}(\ms B) = \sHom_{\Alg{\ms O_Y}}(\phi_\ast \ms O_X, \pi_Y^\ast \ms B) \]
\end{defn}

This functor is not representable in general, as we will see. We show in Theorem~\ref{existence theorem} we show that under certain natural hypotheses we may find a sheaf of $\ms O_S$-algebras $\CAlg{\phi}$ such that $\CFun{\phi}(\ms B) = \sHom_{\Alg{\ms O_S}}(\CAlg\phi, \ms B)$, however in general we do not know that the algebra we construct is locally free over $\ms O_S$. On the other hand, this is clearly true in case $S$ is the spectrum of the field, and we show that, the formation of $\CAlg\phi$ commutes with pullbacks on $S$, showing that the Clifford algebra exists in certain more general settings as well.

Before we discuss the representability of this functor in more general situations, we consider a few examples. Later in \ref{examples}, we will exhibit explicit presentations for the Clifford algebra in various situations.

\begin{example}
Let $k$ be a field, $R$ a commutative $k$-algebra, free and of finite rank over $k$, and $S$ a finitely generated commutative $R$-algebra. Let $\phi: \Spec(S) \to \Spec(R)$ be the structure morphism. Then the Clifford functor of $\phi$ is representable by a $k$-algebra $C$.
\end{example}
\begin{proof}
Let $r_1, \ldots, r_n$ be a basis for $R$ over $k$ with structure coefficients $\lambda_{i,j}^p$ defined by $r_i r_j = \sum_p \lambda_{i,j}^p r_p$, and choose a presentation $S = R[t_1, \ldots, t_m]/(f_1, \ldots, f_\ell)$. We note that for any associative $k$-algebra $B$, a morphism of $R$-algebras $S \to B \otimes R$ is given by a the images of $t_i$ which are given as expressions of the form
$\beta_i = \sum_{j = 1}^n b_{i,j} \otimes r_j$, which must satisfy $f_p(\beta_1, \ldots, \beta_m) = 0$ for $p = 1, \ldots, \ell$. Expanding these expressions and examining each of the coefficients of the $r_q$ gives a number of identities, in the form of noncommutative polynomials $P_{p, q}$ in the $b_{i,j}$ which must be satisfied for this map to be a homomorphism of algebras. The Clifford algebra is then seen to be represented by the free associative algebras in the variables $x_{i,j}$ modulo the ideal generated by the polynomials $P_{p,q}$ in the $x_{i,j}$.
\end{proof}

It is good to keep in mind that when the Clifford algebra exists it need not be nonzero, as the following example shows.
\begin{example}
Let $k$ be a commutative ring, $R = k[x]$, $S = k[x, x^{-1}]$ and $\phi: \Spec(S) \to \Spec(R)$ the structure map. Then the Clifford algebra of $\phi$ exists and is the zero algebra.
\end{example}
\begin{proof}
Suppose that $B$ is an associative $k$-algebra such that we have an $R$-algebra map $S \to B \otimes R$. In particular, we must map $x^{-1}$ to some element $\sum_{i = 0}^n b_i \otimes x^i$ which should satisfy $1 = x \sum b_i \otimes x^i = \sum_{i = 0}^n b_i \otimes x^{i + 1}$. Comparing coefficients, we find that in the coefficient of $x^0$ that we need to have $0 = 1$ in $B$. Hence any such $B$ must receive a map from the $0$-ring, and must therefore be the $0$-ring itself. As all relevant maps are forced to be zero, it must also satisfy the universal property as desired.
\end{proof}

Finally, we give an example to illustrate that the Clifford algebra need not exist in general.
\begin{example}
Let $k$ be a field, $R = k[x]$, $S = k[x, y]$ and $\phi: \Spec(S) \to \Spec(R)$ the structure map. Then the Clifford algebra of $\phi$ does not exist. That is, the Clifford functor is not representable.
\end{example}
\begin{proof}
Suppose $C$ is the Clifford algebra of $\phi$ with $\psi: k[x, y] \to C \otimes_k k[x]$ the morphism associated to the identity . Such a morphism is defined by where $y$ is sent, say $y \mapsto \sum_{i = 0}^n c_i \otimes x^i$.
We claim that we can find some $k$-algebra $B$, together with a map $\gamma: k[x, y] \to B \otimes_k k[x]$ such that there is no algebra map $\rho: C \to B$ such that $\gamma = (\rho \otimes_k k[x]) \psi$.
For this, we set $B = k$ and note that a map $k[x,y] \to B \otimes_k k[x] = k[x]$ of $k[x]$-algebras is given by an arbitrary element of $k[x]$ as an image of $y$, since $k[x,y]$ is a free commutative $k[x]$-algebra with one generator, $y$. Consequently, we can choose $\gamma(y) = x^{n+1}$.
But for any $\rho: C \to B$, we find that $(\rho \otimes_k k[x]) \psi$ takes $y$ to a polynomial in $x$ of degree at most $n$, by definition of $\psi$. Consequently $\gamma$ is not of this form, and $C$ cannot represent the functor as claimed.
\end{proof}

We now describe some useful hypotheses which will help us describe the circumstances under which our Clifford algebras will exist.

\begin{defn}
We say that a sheaf of $\ms O_Y$-modules $\ms N$ is \textit{friendly} with respect to $\pi_Y = \pi$ if it is a locally free coherent sheaf of finite rank, $\pi_* (\ms N^\vee)$ is also locally free of finite rank, and if for any pullback square with $f: S' \to S$ an arbitrary morphism of schemes,
\[\xymatrix@R=.4cm@C=.6cm{
Y' \ar[d]_{\pi'} \ar[r]^{f_Y} & Y \ar[d]^{\pi} \\
S' \ar[r]_f & S
}\]
the base change map gives an isomorphism: $f^\ast \pi_\ast (\ms N^\vee) \cong (\pi')_\ast (f')^\ast (\ms N^\vee)$.
\end{defn}

We note that when $\pi$ is a flat, projective morphism
such that the map $p \mapsto \dim(\Gamma(Y_p, (\ms N^\vee)_p))$ from $S$ to $\mathbb Z$ is locally constant, then by \cite[\S III.9,~Cor~12.9]{Hartshorne} then $\pi_* \mathscr N^\vee$ is locally free on $S$ and the base change map is an isomorphism.
For example, this holds trivially if $S$ is the spectrum of a field.
We note also that the property of being friendly is itself preserved by base change with respect to morphisms $S' \to S$.

\begin{defn}[Condition \weird] \label{weird}
Let $\phi : X \to Y$ be a morphism of $S$-schemes. We say
$\phi$ has condition \weird\ if
$(\phi_\ast \ms O_X)$ and $(\phi_\ast \ms O_X) \otimes_{\ms O_Y} (\phi_\ast \ms O_X)$ are friendly with respect to $\pi_Y$.
\end{defn}

\begin{rem}
  In particular, if $S$ is the spectrum of a field, these conditions will be satisfied when $\phi$ is a finite flat morphism and $Y$ is proper over $S$.
\end{rem}

For a morphism $\phi$ with condition \weird\, an algebra
$\CAlg\phi$ is described in Definition~\ref{clifford definition}, which we will refer to as the Clifford algebra, although in general we are not able to verify our algebra is locally free over $\ms O_S$.

\begin{thm}\label{existence theorem}
If $\phi : X \to Y$ satisfies condition \weird\  then there is a sheaf of $\ms O_S$-algebras $\CAlg\phi$ and an isomorphism of functors from the category $\LFA{\ms O_S}$ to the category of sets.
\[ \sHom_{\Alg{\ms O_S}}(\CAlg\phi, \underline{\ \ }) \cong
\CFun\phi.\]
Further, if $f: S' \to S$ is a morphism of schemes and $\phi': X' \to Y'$ the fiber product of $\phi$ with $S'$ then $\phi'$ also satisfies condition \weird, and we have a canonical identification $f^\ast \CAlg\phi \cong \CAlg{\phi'}$.
\end{thm}
The proof of Theorem \ref{existence theorem} is constructive and occupies Section \ref{sec:construction} below. In particular, for a morphism $\phi$ satisfying condition \weird\ the Clifford algebra $\CAlg\phi$ is constructed in Definition~\ref{clifford definition} below.

\begin{rem}
In particular, if $\CAlg\phi$ is itself locally free (for example, in the
case $S$ is the spectrum of a field), then it represents the functor
$\CFun\phi$, and is the unique locally free algebra (up to canonical isomorphism) that
does so.
\end{rem}

\subsection{Construction of the Clifford algebra}\label{sec:construction}
To define the Clifford algebra, we first describe a construction on
quasi-coherent sheaves, having a somewhat analogous property.

\begin{notation}\label{notation:relative dual}
Given an $S$-scheme $\pi:Y \to S$, and $\ms N$ a sheaf of
$\ms O_Y$-modules, let
\[\ms N_\pi = \left(\pi_\ast \left( \ms N^\vee\right)\right)^\vee = \sHom_{\ms O_S}\bigg(\pi_\ast\big(\sHom_{\ms O_Y}(\ms N, \ms O_Y)\big), \ms O_S\bigg).\]
\end{notation}

The association $\ms N\mapsto\ms N_\pi$ defines a covariant functor from $\ms O_Y$-modules to $\ms O_S$-modules. In general, this functor has no good properties (e.g., it rarely sends quasi-coherent sheaves to quasi-coherent sheaves). We will show that, under certain assumptions, this construction has a universal property and behaves well with respect to base change.

\begin{prop} \label{clifford module}
Suppose that $\pi: Y \to S$
and $\ms N$ is a friendly sheaf of $\ms O_Y$-modules. Then:
\begin{enumerate} [1. ]
\setlength{\itemsep}{.2cm}
\item \label{adjoint}
There is a canonical natural isomorphism of functors (in $\ms M\in\LF{\ms O_S}$)
\[\Xi:
\pi_\ast \sHom_{\ms O_Y} ( \ms N, \pi^\ast\ms M)\to \sHom_{\ms O_S} (\ms N_{\pi}, \ms M).\]
In particular, there is a canonical arrow $\eta:\ms N\to\pi^\ast\ms N_\pi$ such that $\Xi$
sends an arrow $\ms N_\pi \to \ms M$ to the composition
\[ \ms N \overset{\eta}\to \pi^\ast \ms N_\pi \to \pi^\ast \ms M.\]
\item \label{pullback}
Given a pullback square
\[\xymatrix@R=.4cm@C=.6cm{
Y' \ar[d]_{\pi'} \ar[r]^{f_Y} & Y \ar[d]^{\pi} \\
S' \ar[r]_f & S
}\]
we have $f^\ast(\ms N_\pi) = (f_Y^\ast(\ms N))_{\pi'}$.
\end{enumerate}

\end{prop}

Before proving Proposition \ref{clifford module} we require a few lemmas.
\begin{lem} \label{quasicompact global}
Let $X$ be a quasi-compact topological space, $\Lambda$ a cofiltered
category, and $\ms F_\bullet : \Lambda \to \Ab_X$ a cofiltered system of
Abelian sheaves on $X$.  Set $\ms F = \displaystyle\lim_{\to} \ms F_\lambda$. Then
$\Gamma(\ms F, X) = \displaystyle\lim_{\to} \Gamma(\ms F_\lambda, X)$.
\end{lem}
\begin{proof}
By definition, we have that $\ms F$ is the sheafification of the presheaf
which associates to each open set $U$, the set $\displaystyle\lim_{\to} \ms F_\lambda$.
Let $\ms X$ be the category whose elements are open covers of $X$, and with
morphisms $\{U_i \subset X\}_{i \in I} \to \{V_j \subset X\}_{j \in J}$ given by
refinements -- that is by maps of sets $\phi: I \to J$ such that we have
inclusions $U_i \to V_{\phi(i)}$. This is a filtered category, via common
refinements.  By definition of the sheafification, writing $\mc U$ for a
cover $\{U_i\}_{i \in I}$, we have:
\[\ms F(X) = \lim_{\overset{\to}{\mc U \in \ms X}} \ker \left( \prod_{i \in
I} \lim_{\overset{\to}{\lambda}} \ms F_\lambda(U_i) \to \prod_{i,j \in I^2}
\lim_{\overset{\to}{\lambda}} \ms F_\lambda(U_i \cap U_j)\right). \]

Since $X$ is quasi-compact, if we set $\ms X'$ to be the subcategory of $\ms
X$ consisting of finite coverings, we find that
$\ms X'$ is coinitial in $\ms X$, and so we can take limits over $\ms X'$
instead of $\ms X$. In particular, we find that for a cover $\mc U =
\{U_i\}$ in $\ms X'$, products and coproducts (direct sums) coincide over
the finite index sets $I$ and $I^2$. Therefore, we have:

\[\ms F(X) = \lim_{\overset{\to}{\mc U \in \ms X'}} \ker \left(
\bigoplus_{i \in I} \lim_{\overset{\to}{\lambda}} \ms F_\lambda(U_i) \to
\bigoplus_{i,j \in I^2} \lim_{\overset{\to}{\lambda}} \ms F_\lambda(U_i
\cap U_j)\right). \]

In particular, since the direct sum is a colimit, it commutes with the
colimit taken over $\lambda \in \Lambda$. Since the kernel is a finite
limit, it also commutes with the cofiltered colimit in $\lambda$, and
finally, the two colimits described by $\mc U$ and $\lambda$ commute. We
therefore have

\begin{align*}
\ms F(X) &=
\lim_{\overset{\to}{\lambda \in \Lambda}} \lim_{\overset{\to}{\mc U \in \ms X'}} \ker
\left( \bigoplus_{i \in I} \ms F_\lambda(U_i) \to \bigoplus_{i,j \in I^2}
\ms F_\lambda(U_i \cap U_j)\right) \\
&=
\lim_{\overset{\to}{\lambda \in \Lambda}} \lim_{\overset{\to}{\mc U \in \ms X'}} \ker
\left( \prod_{i \in I} \ms F_\lambda(U_i) \to \prod_{i,j \in I^2}
\ms F_\lambda(U_i \cap U_j)\right) \\
&=
\lim_{\overset{\to}{\lambda \in \Lambda}} \lim_{\overset{\to}{\mc U \in \ms X}} \ker
\left( \prod_{i \in I} \ms F_\lambda(U_i) \to \prod_{i,j \in I^2}
\ms F_\lambda(U_i \cap U_j)\right) \\
&=
\lim_{\overset{\to}{\lambda \in \Lambda}} \ms F_\lambda(X)
\end{align*}
where the last equality follows from the fact that $\ms F_\lambda$ is a
sheaf.
\end{proof}

\begin{lem} \label{projection}
Let $\pi : Y \to S$ be a quasi-compact morphism, $\ms F$ a quasi-coherent
sheaf of $\ms O_Y$-modules,
and $\ms G \in \LF{\ms O_S}$. Then the natural morphism of sheaves of $\ms
O_S$-modules:
\[ \pi_\ast(\ms F) \otimes_{\ms O_S} \ms G \to \pi_\ast(\ms F \otimes_{\ms O_Y}
\pi^\ast \ms G) \]
is an isomorphism.
\end{lem}
\begin{proof}
Since tensor product and pushforward commute with flat base change, we may work locally on $S$ and assume that
\[\ms G = \ms O_S^{\oplus I}\] for some index set $I$ (the exponent indicating direct sum indexed by the elements of $I$). We have
\[ \pi_\ast(\ms F) \otimes_{\ms O_S} \ms G  =
\pi_\ast(\ms F) \otimes_{\ms O_S} \ms O_S^{\oplus I}  =
\pi_\ast(\ms F) ^{\oplus I}\]

On the other hand, we have:
\[ \pi_\ast(\ms F \otimes_{\ms O_Y} \pi^\ast \ms G) = \pi_\ast( \ms F \otimes_{\ms
O_Y} \pi^\ast \ms O_S^{\oplus I})\]
Since tensor and direct sum commute, we have
\[\pi_\ast( \ms F \otimes_{\ms O_Y} \pi^\ast \ms O_S^{\oplus I}) =
\pi_\ast( \ms F \otimes_{\ms O_Y} \ms O_Y^{\oplus I}) = \pi_\ast(\ms F^{\oplus I})
\]

But since we can write the direct sum as a cofiltered colimit of finite
direct sums, by Lemma~\ref{quasicompact global} we can identify $\pi_\ast(\ms
F)^{\oplus I}$ with $\pi_\ast(\ms F^{\oplus I})$, completing the proof.

\end{proof}

\begin{proof}[Proof of Proposition~\ref{clifford module}]

For part~\ref{adjoint}, using the fact that $\ms N$ is finite locally free, we have
\begin{align*}
\pi_\ast \sHom_{\ms O_Y} ( \ms N, \pi^\ast\ms M) &\cong
\pi_\ast \sHom_{\ms O_Y} ( \ms O_Y, \ms N^\vee \otimes_{\ms O_Y} \pi^\ast\ms M) \\
&\cong
\pi_\ast \left(\ms N^\vee \otimes_{\ms O_Y} \pi^\ast\ms M\right)
\end{align*}
and, since $\ms M$ is locally free, we have, by Lemma~\ref{projection} that
\[ \pi_\ast \left(\ms N^\vee \otimes_{\ms O_Y} \pi^\ast\ms M\right) \cong \pi_\ast (\ms N^\vee) \otimes_{\ms
O_S}\ms M.\]
Finally, since $\pi_\ast(\ms N^\vee)$ is finite locally free, we have
\begin{align*}
\pi_\ast (\ms N^\vee) \otimes_{\ms O_S}\ms M &\cong \sHom_{\ms O_S}\left( \ms O_S,  \pi_\ast (\ms N^\vee) \otimes_{\ms O_S}\ms M\right)  \\
&\cong \sHom_{\ms O_S}\left( (\pi_\ast (\ms
N^\vee))^\vee,\ms M\right)
\\ &\cong \sHom_{\ms O_S} ( \ms N_{\pi},\ms M),
\end{align*}
as desired.

For part~\ref{pullback}, as $\ms N$ is friendly with respect to $\pi$, we have
\begin{align*}
f^\ast\pi_\ast\left(\ms N^\vee\right)
&\cong \pi'_\ast {f_Y}^\ast \left(\ms N^\vee\right) \\
&\cong \pi'_\ast {f_Y}^\ast\sHom_{\ms O_Y}\left(\ms N, \ms O_Y \right) \\
\end{align*}
and since $\ms N$ is finite locally free, we have
\begin{align*}
\pi'_\ast {f_Y}^\ast\sHom_{\ms O_Y}\left(\ms N, \ms O_Y \right)
&\cong \pi'_\ast\sHom_{\ms O_{Y'}}\left(f_Y ^\ast \ms N, \ms O_{Y'} \right) \\
&=\pi'_\ast\left(f_Y^\ast(\ms N)^\vee\right).
\end{align*}
Since $\pi_\ast(\ms N^\vee)$ is finite locally free,
\begin{align*}
f^\ast (\ms N_\pi) = f^\ast\left(\pi_\ast (\ms N^\vee)\right)^\vee
&= f^\ast \sHom_{\ms O_S} \left(\pi_\ast (\ms N^\vee), \ms O_S\right) \\
&= \sHom_{\ms O_{S'}}\left(f^\ast(\pi_\ast (\ms N^\vee)), \ms O_{S'}\right) \\
&= \sHom_{\ms O_{S'}}\left(\pi'_\ast((f_Y^\ast(\ms N))^\vee), \ms O_{S'}\right) \\
&= \pi'_\ast\left((f_Y^\ast \ms N)^\vee\right)^\vee \\
&= \left(f_Y^\ast \ms N\right)_{\pi'}
\end{align*}
as claimed.

\end{proof}

To use this module in the construction of the Clifford algebra, we first
introduce a ``relative free algebra construction:''

\begin{defn} \label{unital modules}
A \textit{unital $\ms O_S$-module} is a
quasi-coherent sheaf of $\ms O_S$-modules $\ms
N$, together with a $\ms O_S$-module morphism
$\epsilon_{\ms N} : \ms O_S \to \ms N$, referred
to as the \textit{unit morphism}. A morphism of
unital $\ms O_S$-modules is simply an $\ms
O_S$-module morphism which commutes with the unit
morphisms. We let $\UMod{\ms O_S}$ denote the
category of unital $\ms O_S$-modules.
\end{defn}

\begin{notation}\label{unitalization}
For any scheme $Z$, let
$$\Unitalization:\Alg{\ms O_Z}\to\UMod{\ms O_Z}$$ denote the canonical forgetful functor that sends an $\ms O_Z$-algebra $\ms A$ to the unital module given by the underlying $\ms O_Z$-module of $\ms A$ together with the identity element $\ms O_Z\to\ms A$.
\end{notation}

The functor $\Unitalization$ has a left adjoint.
We note that there is a forgetful map from the category of quasi-coherent
sheaves of $\ms O_S$-algebras to the category of unital $\ms O_S$-modules.
The left adjoint to this is constructed as follows:

\begin{lem} \label{free algebra}
Let $S$ be a scheme. There is a ``free algebra''
functor
$$\ms N, \epsilon\mapsto\Free{\ms
N,\epsilon}:\UMod{\ms O_S}\to\Alg{\ms O_S}$$
that is left adjoint to $\Unitalization$.
Moreover, for a morphism $f : T \to S$, we have
$$f^\ast \Free{\ms N, \epsilon} = \Free{f^\ast \ms N,
f^\ast\epsilon}.$$
\end{lem}
We will usually omit $\epsilon$ from the notation and write $\Free{\ms N}$.
\begin{rem} \label{free counit}
The counit of the adjunction yields a
canonical morphism
$$\xi : \ms N \to \Unitalization\Free{\ms N}$$ of
unital $\ms O_S$-modules
such that a morphism of $\ms O_S$-algebras
$\Free{\ms N} \to \ms B$ is
associated to the composition of $\ms O_S$-module
maps $\ms N \to \Unitalization\Free{\ms N}
\to\Unitalization\ms B$.
\end{rem}
\begin{proof}
We construct $\Free{\ms N}$ as the sheafification of a presheaf as follows.
For an affine open set $U = \Spec R \subset S$, write $\ms N(U) = N$, and
$\iota(U) = i : R \to N$. We consider the algebra $\Free{N}$ to be the free
associative $R$-algebra $R\left<N\right>$ (the tensor algebra) modulo the
ideal $I$ generated by the expressions of the form $i(r) - r$ where $r \in
R$, the element $r$ being viewed on the right as taken from the
coefficients of the tensor algebra. It is clear that this presheaf of algebras
has the corresponding universal property among presheaves of algebras, and
hence by the universal property of sheafification the resulting sheafified
algebra has the correct universal property as well.

The assertion concerning the behavior under pullback will follows from the
fact that, on the level of affine schemes, this description is preserved by
tensor products with respect to a homomorphism of rings $R \to R'$ and the
construction of the tensor algebra commutes with
base change to
$R'$.

\end{proof}

In classical constructions of the Clifford algebra of a homogeneous form
(see, for example \ref{examples}), the Clifford algebra is defined
as a free associative algebra, generated by an underlying vector space of
the form, modulo a certain ideal. The
construction above gives an analog of this free
algebra; we will now describe the construction of
the corresponding ideal in the relative case.

\begin{defn}
Let $\pi : Y \to S$ be a morphism of schemes, $\ms A$ a quasi-coherent sheaf
of $\ms O_Y$-algebras. An \textit{agreeable algebra for $\ms A$} is a
quasi-coherent sheaf of $\ms O_S$-algebras $\ms B$ together with a morphism
$\upsilon_{\ms B} : \ms A \to \pi^\ast \ms B$ of unital $\ms O_Y$-modules in
the sense of Definition~\ref{unital modules}.
A morphism of agreeable algebras is a morphism of sheaves of algebras $f : \ms
B \to \ms D$ such that we have a commutative diagram
\[\xymatrix@R=.3cm@C=2cm{
 & \pi^\ast \ms B \ar[dd]^{\pi^\ast f} \\
\ms A \ar[ru]^{\upsilon_{\ms B}} \ar[rd]_{\upsilon_{\ms D}} \\
 & \pi^\ast \ms D.
}\]
We let $\Agree{\ms A}$ denote the
category of agreeable algebras for $\ms A$. We say that an agreeable map $f: \ms B \to \ms D$ is a \textit{compromise} for $(\ms B, \upsilon_{\ms B})$ if $\upsilon_{\ms D}: \ms A \to \pi^{\ast} \ms D$ is a morphism of $\ms O_Y$-algebras. We write $\Comp_{\ms A}(\ms B, \ms D)$ for the set of compromises from $(\ms B, \upsilon_{\ms B})$ to $\ms D$.
\end{defn}

\begin{lem} \label{freaky ideal sheaf}
Let $\pi : Y \to S$ be a morphism of schemes, $\ms A$ a sheaf of $\ms O_Y$-algebras, such that $\ms A$ and $\ms A \otimes_{\ms O_Y} \ms A$ are friendly with respect to $\pi$.
Let $\ms B$,
$\upsilon: \ms A \to \pi^\ast \ms B$ an agreeable algebra for $\ms A$. Then
there is a sheaf of ideals $\freakyideal\upsilon \triangleleft \ms B$ such
that
\begin{enumerate}[1. ]
\item \label{freaky agreeable}
the induced map $\ms A \to \pi^\ast \left(\ms B/\freakyideal\upsilon\right)$ is a map of $\ms O_Y$-algebras, that is, it is a compromise for $(\ms B, \upsilon)$,
\item \label{freaky smallest}
for any other quasi-coherent sheaf of ideals $\ms J \triangleleft \ms B$
such that
the induced map $\ms A \to \pi^\ast \left(\ms B/\ms J\right)$ is a map of $\ms O_Y$-algebras,
$\freakyideal\upsilon \subset \ms J$,
\item \label{freaky adjoint}
for any quasi-coherent sheaf of $\ms O_S$-algebras $\ms D$, we have a
natural bijection
\[
\Comp_{\ms A}(\ms B, \ms D) = Hom_{\Alg{\ms
O_S}}(\ms B/\freakyideal\upsilon, \ms D)\]
\item \label{freaky natural}
for $f: T \to S$ a morphism of schemes, we have $f^\ast \left(\ms
B/\freakyideal\upsilon\right) \cong (f^\ast\ms B)/\freakyideal{f^\ast \upsilon}$
\end{enumerate}
\end{lem}
\begin{proof}[Proof of Lemma~\ref{freaky ideal sheaf}]
Let $(\ms D, \upsilon_{\ms D})$ be an agreeable algebra for $\ms A$.
Via the multiplication maps on $\ms D$ and $\ms A$, we may define two maps of $\ms O_Y$-modules:
\begin{align*}
  m_1^{\ms D}: \ms A \otimes_{\ms O_Y} \ms A &\overset{m_{\ms A}}\longrightarrow \ms A \overset{\upsilon_{\ms D}}\to \pi^\ast \ms D \\
  m_2^{\ms D}: \ms A \otimes_{\ms O_Y} \ms A &\overset{\upsilon_{\ms D} \otimes \upsilon_{\ms D}}\longrightarrow \pi^\ast \ms D \otimes_{\ms O_Y} \pi^\ast \ms D =
  \pi^\ast(\ms D \otimes_{\ms O_S} \ms D) \overset{\pi^\ast m_D}\longrightarrow \pi^\ast \ms D
\end{align*}
and by definition, $\upsilon_{\ms D}$ is a map of $\ms O_Y$-algebras if $m_1^{\ms D} = m_2^{\ms D}$. Consider the difference
$\delta_{\ms D} = m_1^{\ms D} - m_2^{\ms D} \in \sHom_{\ms O_Y}(\ms A \otimes_{\ms O_Y} \ms A, \pi^\ast {\ms D})$.
It follows from the construction that this is natural in the sense that if we are given $f: (\ms D, \upsilon_{\ms D}) \to (\ms D', \upsilon_{\ms D'})$, a morphism of agreeable algebras for $\ms A$,
then we have a commutative diagram
\[\xymatrix{
\ms A \otimes \ms A \ar[r]^{\delta_{\ms D}} \ar[rd]_{\delta_{\ms D'}} & \pi^\ast \ms D \ar[d]^{\pi^\ast f} \\
 & \pi^\ast \ms D'
}\]
and we see that a morphism $(\ms D, \upsilon_{\ms D}) \to (\ms D', \upsilon_{\ms D'})$ is a compromise (i.e., $\upsilon_{\ms D'}$ is an $\ms O_Y$-algebra map) if and only if $\delta_{\ms D'} = 0$.

Since $\ms A$ and $\ms A \otimes_{\ms O_Y} \ms A$ are both friendly, given a morphism as before $f: (\ms D, \upsilon_{\ms D}) \to (\ms D', \upsilon_{\ms D'})$ of agreeable algebras for $\ms A$, the diagram above corresponds to a diagram of $\ms O_S$-modules:
\begin{equation}\label{freaky ideal natural}
  \xymatrix{
(\ms A \otimes \ms A)_\pi \ar[r]^-{(\delta_{\ms D})_\pi} \ar[rd]_-{(\delta_{\ms D'})_\pi} & \ms D \ar[d]^{ f} \\
 & \ms D'
}
\end{equation}
and in particular, $\upsilon_{\ms D}$ is an $\ms O_Y$-algebra map if and only if the map $(\delta_{\ms D})_\pi$ is the zero map. Define $\freakyideal{\upsilon_{\ms D}}$ to be the sheaf is ideals generated by the image of $(\delta_{\ms D})_\pi$. It follows from that construction that for any morphism $f: (\ms D, \upsilon_{\ms D}) \to (\ms D', \upsilon_{\ms D'})$ of agreeable algebras, we have
$f(\freakyideal{\upsilon_{\ms D}}) \subset \freakyideal{\upsilon_{\ms D'}}$.
In particular, we find that considering the quotient map $(\ms B, \upsilon) \to (\ms B/\freakyideal{\upsilon}, \overline{\upsilon})$ we obtain Part~\ref{freaky adjoint}, and consequently
this quotient map is a compromise, verifying part~\ref{freaky agreeable}.

Part~\ref{freaky smallest} follows from the fact that for an ideal sheaf $\ms J \triangleleft \ms B$, we have seen that the induced map $(\ms B, \upsilon) \to (\ms B/\ms J, \til\upsilon)$ is a compromise if and only if $(\delta_{\til\upsilon})_\pi$ is zero,
which by the commutativity of diagram~\eqref{freaky ideal natural} happens if and only if $\freakyideal{\upsilon} \subset \ms J$ as desired.

Finally, part~\ref{freaky natural} follows from the fact that the formation of the image sheaf theoretic image commutes with base change, as does the operation of an ideal sheaf generated by a subsheaf.
\end{proof}

We are now prepared to give the definition of the Clifford algebra.
\begin{defn} \label{clifford definition}
Let $\phi: X \to Y$ be a morphism of $S$-schemes satisfying condition
\weird (see Notation \ref{weird}). By Proposition~\ref{clifford module}(\ref{adjoint}), we have a
morphism \[\eta: \phi_\ast \ms O_X \to \pi^\ast \left(\left(\phi_\ast \ms
O_X\right)_\pi\right)\]
Which we can consider as a map of unital $\ms O_Y$-modules via the structure map $\ms O_Y \to \phi_* \ms O_X$.
Consider the counit map
\[\xi: \pi^\ast\left( \left(\phi_\ast \ms O_X\right)_\pi\right) \subset
\Free{\pi^\ast\left( \left(\phi_\ast \ms O_X\right)_\pi\right)} = \pi^\ast
\Free{\left(\phi_\ast \ms O_X\right)_\pi}\] of
Remark~\ref{free counit}, and let $\upsilon_\phi$ be the composition:
\[\xymatrix{
\phi_\ast \ms O_X \ar[r]_-{\eta} \ar@/^1.5pc/[rr]^{\upsilon_\phi} & \pi^\ast
\left(\phi_\ast \ms O_X\right)_\pi \ar[r]_-{\xi} &
\pi^\ast \Free{\left(\phi_\ast \ms O_X\right)_\pi}
}\]
Letting $\freakyideal{\upsilon_{\phi}} \triangleleft \Free{\left(\phi_\ast \ms
O_X\right)_\pi}$ be the ideal as defined in Lemma~\ref{freaky
ideal sheaf}, we
then define the Clifford algebra $\CAlg\phi$ to be the sheaf of $\ms
O_S$-algebras given by
\[ \CAlg\phi =
\Free{\left(\phi_\ast \ms O_X\right)_\pi}/ \freakyideal{\upsilon_\phi}\]
\end{defn}

We now show that this construction behaves well with respect to pullbacks.

\begin{lem} \label{clifford base change}
Let $\phi: X \to Y$ be a morphism of $S$-schemes satisfying condition
\weird.  If $f: T \to S$ is any morphism, and $\phi_T : X_T \to Y_T$ the
pullback morphism, then there is a natural isomorphism of sheaves of
algebras
\[f^\ast \CAlg\phi \cong \CAlg{\phi_T}.\]
\end{lem}
\begin{proof}
By Lemma~\ref{free algebra}, we have \[f^\ast \Free{\left(\phi_\ast \ms
O_X\right)_\pi} = \Free{f^\ast \left(\left(\phi_\ast\ms O_X\right)_\pi\right)}.\]
Since the hypotheses imply that $(\phi_\ast \ms O_X)_\pi$ is finite and
locally free, it follows from Proposition~\ref{clifford module}(\ref{pullback}) that \[f^\ast \Free{\left(\phi_\ast \ms
O_X\right)_\pi}
=
\Free{\left((\phi_T)_\ast \ms O_{X_T}\right)}.\] The result now follows from
Lemma~\ref{freaky ideal sheaf}(\ref{freaky natural}).
\end{proof}

We now prove that the algebra constructed above
has the desired properties.
\begin{proof}[Proof of Theorem \ref{existence
theorem}]
The fact that the hypotheses are preserved by base change follows from the analogous fact for friendly sheaves,
and the fact that the construction is preserved by base change follows from Lemma~\ref{clifford base change}.
Without loss of generality, via changing the base, we may reduce to the case $S$ is affine and hence check this
by simply comparing global sections on each side. We have
\[\Hom_{\ms O_Y\text{-alg}}(\phi_\ast \ms O_X, \pi^\ast \ms B) =
\{ \rho \in \Hom_{\ms O_Y}(\phi_\ast \ms O_X, \pi^\ast \ms B) \mid \rho \text{ an
alg.\ hom.}\}\]
and using Proposition~\ref{clifford module}(\ref{adjoint}), this is identified
with
\[\{ \psi \in \Hom_{\ms O_S}((\phi_\ast \ms O_X)_\pi, \ms B) \mid \ms \phi_\ast
\ms O_X \overset\eta\to \pi^\ast((\phi_\ast\ms O_X)_\pi) \overset{\psi}{\to}
\pi^\ast \ms B  \text{ an alg.\ hom.}\}\]
using Lemma~\ref{free algebra}, this is in bijection with
\[\{ \til\psi \in \Hom_{\ms O_S\text{-alg}}(\Free{(\phi_\ast \ms O_X)_\pi}, \ms B)
\mid \ms \phi_\ast \ms O_X \overset{\upsilon_\phi}\to \Free{\pi^\ast((\phi_\ast\ms
O_X)_\pi)} \overset{\psi}{\to} \pi^\ast \ms B  \text{ an alg.\ hom.}\}\]
which we can identify with $\Comp_{\phi_\ast \ms O_X}(\Free{\pi^\ast((\phi_\ast\ms
O_X)_\pi)}, \ms B)$.
Finally, by Lemma~\ref{freaky ideal sheaf}(\ref{freaky adjoint}), this is
in bijection with
\[
\Hom_{\ms O_S\text{-alg}}(\Free{\pi^\ast((\phi_\ast \ms O_X)_\pi)}/\ms
I_{\upsilon_\phi}, \ms B) = \Hom_{\ms O_S\text{-alg}}(\CAlg\phi, \ms B)\]
as desired.
\end{proof}

\section{Representations and Ulrich bundles} \label{ulrich}

Having discussed the Clifford algebra, in this section we describe its
representations. Since we will be considering a morphism of $S$-schemes
$\phi: X \to Y$ and representations on sheaves of $\ms O_S$-modules, there
are pullback functors on categories of representations induced by base
changes $T \to S$. Consequently, the various categories of representations,
as the base changes, will fit together into a stack, and this will be the
natural way to describe the arithmetic and geometry of these
representations.

For considering representations of Clifford algebras, again we fix a base
scheme $S$, and it will be useful to restrict to only certain morphisms
$\phi : X \to Y$ of $S$-schemes. We strengthen condition \weird\ to also
require our base variety to be proper and have connected fibers:
\begin{defn} \label{wweird}
Let $\phi : X \to Y$ be a morphism of $S$-schemes. We say that $\phi$ has
condition \wweird\ if
\begin{enumerate}[\ \ 1. ]
\item
$\pi_Y$ is proper, flat, and of finite presentation;
\item
$\phi$ is finite locally free and surjective;
\item
$\ms O_Y$, $\phi_\ast \ms O_X$ and $\phi_\ast \ms O_X \otimes_{\ms O_Y} \phi_\ast \ms O_X$ are all friendly with respect to $\pi_Y$.
\end{enumerate}
\end{defn}

Note that $\ms O_Y$ being friendly with respect to $\pi_Y$ would be implied by $\pi_Y$ being cohomologically flat in dimension $0$.
In applications, $S$ will often be the spectrum of a field and $Y$ will be a geometrically integral proper variety over $S$.

\begin{defn}
Let $S$ be a scheme, and $\ms A$ a sheaf of $\ms O_S$-algebras. We define
$\Rep{\ms A}{}$ to be the category whose objects are pairs $(f, W)$ where
$W$ is a locally free sheaf of $\ms O_S$-modules of finite rank, and where
$f : \ms A \to \send_{\ms O_S}(W)$ is a homomorphism of $\ms
O_S$-algebras. A morphism $(f, W) \to (g, U)$ in $\Rep{\ms A}{}$ is a morphism of $\ms
O_S$-modules $\lambda: W \to U$ such $\lambda(f(a)(w)) = g(a)(\lambda(w))$ for sections $w, a$ of $W$ and $\ms A$ respectively. We let $\Rep{\ms A}{n}$ denote the subcategory of pairs $(f, W)$
where $W$ has rank $n$ over $\ms O_S$.
\end{defn}

From these, we obtain stacks $\stRep{\ms A}{}$ (respectively $\stRep{\ms
A}{n}$) defined over $S$ which associates to $U \to S$ the (isomorphisms in
the) category $\Rep{\ms A_U}{}$ (respectively $\Rep{\ms A_U}{n}$).
In the case of a morphism $\phi: X \to Y$ of $S$-schemes, we simply write
$\Rep\phi{}$ and $\Rep\phi{n}$ to denote $\Rep{\CAlg\phi}{}$ and
$\Rep{\CAlg\phi}{n}$ respectively.

\begin{defn}
Let $\phi : X \to Y$ be a morphism of $S$-schemes. We say that a coherent
sheaf $V$ of $\ms O_X$-modules is \emph{Ulrich for $\phi$} if there is a fppf covering $S'\to S$ (that is, a surjective flat map, locally of finite presentation) and a section $r\in H^0(S', \Z)$ such that \[\phi_\ast V_{S'}\cong\pi_{Y_{S'}}^\ast\ms O_{S'}^{\oplus r}.\] The Ulrich sheaves
(respectively, the Ulrich sheaves with pushforward of rank $m$ for a fixed integer $m$) form a full subcategory of
the category of coherent sheaves on $X$, which we denote by $\Ulr\phi{}$
(respectively $\Ulr\phi{m}$).
\end{defn}

Ulrich sheaves form a substack $\stUlr\phi{}$ of the
$S$-stack of coherent sheaves on $X$ by setting $\stUlr\phi{}(U) = \Ulr{\phi_U}{}$,
for $U$ an $S$-scheme. (Of course, one should add additional hypotheses, such as flatness over the base, if one hopes to produce an algebraic stack but this is inessential here.)

The following proposition is a generalization of \cite[Proposition~1]{VdB:Lin}.
\begin{prop} \label{Rep equal Ulr}
Let $\phi: X \to Y$ be a morphism satisfying condition \wweird. Then
there is an isomorphism of stacks
$\theta: \stRep{\phi}{} \to \stUlr\phi{}$, inducing equivalences of
categories $\Rep{\phi}{} \cong \Ulr\phi{}$. In the case that $\phi$ has
constant degree $d$, this gives for every positive integer $m$ an
equivalence $\Rep\phi{dm} \cong \Ulr\phi{m}$. In particular, every
representation of the Clifford algebra has rank a multiple of $d$.
\end{prop}

\begin{lem} \label{proper to affine}
Suppose that $\pi: Y \to S$ is a morphism of schemes such that the natural map $$\ms O_S\to\pi_\ast\ms O_Y$$ is an isomorphism. Given an affine morphism $H \to
S$, the natural map
\[ Hom_{S}(S, H) = Hom_{S}(Y, H) \]
is bijective.
\end{lem}
\begin{proof}
This is an immediate consequence of \cite[Ch.~I, Prop.~2.2.4]{EGA1}.
\end{proof}

\begin{proof}[Proof of Proposition~\ref{Rep equal Ulr}]
By the universal property of the Clifford algebra, it is easy to see that
$\stRep\phi{}$ is equivalent to the stack whose objects over $T$ consist of
pairs $(W, \psi)$ where $W$ is a finite locally free sheaf of $\ms
O_T$-modules, and where $\psi : \phi_\ast \ms O_{X_T} \to \send_{\ms O_{Y_T}}(\pi^\ast
W)$ is a $\phi_\ast \ms O_{X_T}$-module structure on $\pi^\ast W$, and where
morphisms must preserve the $\pi_\ast \ms O_{X_T}$-module structure. This in turn,
is equivalent to the stack whose objects over $T$ consist of triples $(W, V, f)$, where
\begin{itemize}
\item $W$ is a finite locally free sheaf of $\ms O_T$-modules,
\item $V$ is a coherent sheaf of $\ms O_{X_T}$-modules,
\item $f : \phi_\ast V \to \pi^\ast W$ is an isomorphism of $\ms
O_{Y_T}$-modules, and
\item morphisms $(W, V, f) \to (W', V', f')$ are given by maps
$\alpha: W
\to W'$ and $\beta: V \to V'$ such that the diagram
\[\xymatrix{
\phi_\ast V \ar[d]_{f} \ar[r]^{\phi_\ast \beta} & \phi_\ast V' \ar[d]^{f'} \\
\pi^\ast W \ar[r]_{\pi^\ast\alpha} & \pi^\ast W'
}\]
commutes.

\end{itemize}

We define $\theta: \stRep\phi{} \to \stUlr\phi{}$ by sending $(W, V, f)$ to $V$.
Essential surjectivity of $\theta$ follows from the definition of the Ulrich
stack.

Now consider $\stRep\phi{}$ as a fibered category over
$\stUlr\phi{}$. To see that we have an equivalence of stacks, it
suffices to show that, for an object $V \in \stUlr\phi{}$, that
the fiber category over $V$ is a groupoid such that for every pair of
objects $a, b \in \stRep\phi{}$, the morphism set
$Hom_{\stRep\phi{}(V)}(a, b)$ consists of a single element. To
verify this, we choose $a = (W_1, V, \alpha_1)$, $b = (W_2, V, \alpha_2)$.
Define $\gamma : \pi^\ast W_1 \to \pi^\ast W_2$ as the unique isomorphism of $\ms
O_Y$-modules making the diagram
\[\xymatrix{
& V \ar[ld]_{\alpha_1} \ar[rd]^{\alpha^2} \\
\pi^\ast W_1 \ar[rr]_{\gamma} & & \pi^\ast W_2
}\]
commute. We claim that the map $\gamma$ comes from a unique isomorphism
$\til \gamma : W_1 \to W_2$. To do this, we consider the sheaf $\ms H = \sHom_{\ms O_S}(W_1, W_2)$.
Local freeness of the $W_i$s implies that we have $\pi^\ast \ms H = \ms
Hom_{\ms O_Y}(\pi^\ast W_1, \pi^\ast W_2)$. Setting $H \to S$ to be the
underlying affine scheme of the vector bundle $\ms H$, we then have $H
\times_S Y$ is the underlying affine scheme of the vector bundle $\pi^\ast \ms
H$. We would like to show that a given section $\gamma : Y \to H \times_S
Y$ comes via pullback from a unique section $\til \gamma: S \to H$. But
this is exactly Lemma~\ref{proper to affine}
\end{proof}

We can similarly define projectively Ulrich bundles and Azumaya
representations as follows.
\begin{defn}
Let $S$ be a scheme, and $\ms A$ a sheaf of $\ms O_S$-algebras. The category of \emph{Azumaya representations\/}, denoted
$\AzRep{\ms A}{}$ (respectively, \emph{Azumaya representations of degree $n$\/}, denoted $\AzRep{\ms A}{n}$) is the category
whose objects are pairs $(f, B)$ where $B$ is a sheaf of Azumaya algebras
over $S$ (respectively, Azumaya algebras over $S$ of degree $n$) ,
and where $f : \ms A \to B$ is an isomorphism of $\ms O_S$-algebras. A
morphism $(f, B) \to (g, C)$ will be a morphism of $\ms O_S$-algebras $B
\to C$ such that the diagram
\[\xymatrix@R=.2cm @C=2cm{
& B \ar[dd] \\
\ms A \ar[ru]^f \ar[rd]_g \\
& C
}\]
commutes.
\end{defn}

As before, we obtain a stack $\stAzRep{\ms A}{}$ (respectively
$\stAzRep{\ms A}{n}$) defined over $S$ which associates to $U \to S$ the
category $\AzRep{\ms A_U}{}$ (respectively $\stAzRep{\ms A}{n}$). This
stack carries a universal sheaf of Azumaya algebras $\mf A_{\ms A}$ and
a universal representation $\ms A_{\stAzRep{\ms A}{}} \to \mf A_{\ms A}$. We note
that there is a natural morphism of stacks $\stRep{\ms A}{} \to
\stAzRep{\ms A}{}$ which gives $\stRep{\ms A}{}$ the structure of a $\mathbb
G_m$-gerbe over $\stAzRep{\ms A}{}$, whose class over an object $(f, B)$ is
precisely the Brauer class of $B$. In particular, the global class of the
gerbe is given by the algebra $\mf A_{\ms A}$.

\begin{defn}
Let $\phi: X \to Y$ be morphisms of $S$-schemes.  The \emph{category of
projectively Ulrich bundles\/} $\PrUlr\phi{}$ is the category of sheaves of modules $V$ on
$X$ such that the projective bundle $\mbb P(\phi_\ast V)\to Y$ is isomorphic over $Y$ to $P\times_S Y\to Y$, where $P\to S$ is a Brauer-Severi scheme.
\end{defn}
This is to say, we require that $P$ is isomorphic fppf-locally on $S$ to $\mbb P^n_S$.

Comparing automorphism groups, for a morphism $\phi: X \to Y$ satisfying
condition \wweird, we have a diagram of
stacks that commutes up to $2$-isomorphism
\[\xymatrix{
\stRep{\phi}{md} \ar[r] \ar[d] & \stAzRep\phi{md} \ar[d] \\
\stUlr\phi{m} \ar[r] & \stPrUlr\phi{m}
}\]
in which the vertical arrows are equivalences.

\begin{rem}
Note that the horizontal arrows in the above diagram are $\mbb G_m$-gerbes, so the stacks in each row of the diagram have isomorphic sheafifications.
\end{rem}

\begin{defn}
We say that a representation $(f, W) \in \Rep\phi{}$ is a
\emph{specialization} if $f : \CAlg\phi \to \send(W)$ is surjective (and
similarly for objects in $\AzRep\phi{}$).
\end{defn}

Recall that a vector bundle $V/X$ is called \textbf{simple} if its
automorphism sheaf over $S$ is $\mbb G_m$.
\begin{defn}
Let $\SUlr\phi{}, \PrSUlr\phi{}$ denote the categories of simple Ulrich
bundles and projectively bundles respectively. Let $\stSUlr\phi{},
\stPrSUlr\phi{}$ denote the associated substacks of $\stUlr\phi{},
\stPrUlr\phi{}$.
\end{defn}

\begin{rem} \label{special simple}
If a representation $(f, W)$ is a specialization, then it follows that its
associated Ulrich bundle is simple (from the fact that any automorphism of
a vector space which commutes with every linear transformation must be
central and hence scalar multiplication).
\end{rem}

\begin{rem} \label{simple gerby}
Since the objects of $\SUlr\phi{}$ have automorphism group $\mbb G_m$, it
follows that we can identify $\PrSUlr\phi{}$ with the sheafification of
$\SUlr\phi{}$.
\end{rem}

To examine the representations of the Clifford algebra, it is natural to consider imposing the identites of $d \times d$ matrices. Let us first recall how this is done, following \cite[Prop~1.3.10]{RowenPI}. Consider the polynomial ring $\mathbb Z[\xi]$ in countably many (commutative) indeterminates
$\xi_{i,j}^\ell$ and let $m^{(\ell)} \in M_n(\mathbb Z[\xi])$ be the matrix with entries $(\xi_{i,j}^\ell)$. Consider
$\mathbb Z\langle \vec a \rangle$
the free associative algebra in countably many (noncommutative) indeterminates $a_k$. Let $\til I_d$ denote the kernel of the map
$\mathbb Z\langle \vec a \rangle \to \mathbb Z_n[\xi]$ defined by sending $a_\ell$ to the ``generic matrix'' $m^{(\ell)}$.

\begin{defn}
For $p \in \til I_d$, and for $\vec b$ a tuple of elements in $B$, we consider the expression $p(\vec b)$ in
$B$ obtained by specializing $a_\ell$ to $b_\ell$.
We define $I_d(B)$, the \textit{ideal of identities of $d \times d$ matrices in $B$}, be the ideal of $B$ generated by all elements of the form $p(\vec b)$ for $p \in \til I_d$ and all tuples $\vec b \in B^{\mathbb N}$.
\end{defn}

\begin{defn} \label{reduced clifford}
Let $\phi: X \to Y$ be a morphism of $S$-schemes satisfying condition
\wweird, with $\phi$ constant rank $d$. We let the \textbf{reduced Clifford
algebra}, denoted $\CAlg\phi^{\text{red}}$, be
the quotient of $\CAlg\phi$ by the sheaf of ideals $I_d(\CAlg\phi)$ generated locally by all
the identities of $d \times d$ matrices.
\end{defn}

\begin{thm} \label{good azumaya love}
Let $\phi: X \to Y$ be a morphism of $S$-schemes satisfying condition
\wweird, with $\phi$ of constant degree $d$.  Let $\ms C =
\CAlg\phi^{\text{red}}$, and $\ms Z = Z(\ms C)$ its center. Then
\begin{enumerate}
\item \label{identities don't hurt the small guy}
there is an equivalence of categories $\Rep\phi{d} \cong \Rep{\ms C}d$ that is functorial with respect to base change on $S$,
inducing an isomorphism of stacks $\stRep\phi{d} \cong \stRep{\ms C}d$.
\end{enumerate}
When $S = \Spec k$ is the spectrum of a field, we have in addition that
\begin{enumerate}
\setcounter{enumi}{1}
\item \label{field minimal azumaya} $\ms C$ is Azumaya over $\ms Z$ of rank
$d$;
\item \label{field azumaya center}
there is a natural isomorphism $\Spec(\ms Z) \cong \stAzRep\phi{d}$, and
every Azumaya representation of degree $d$ is a specialization;
\item \label{field small specializations}
if $v \in \stAzRep\phi{d}$ then the class of the gerbe $\stRep\phi{d} \to
\stAzRep\phi{d}$ lying over $v$ is exactly $\ms C \otimes_{\ms Z} k(v)$,
where $k(v)$ is the residue field of $v$.
\end{enumerate}
\end{thm}
\begin{proof}
Part~\ref{identities don't hurt the small guy} follow immediately from the
fact that any $d \times d$ identities are automatically in the kernel of
these representations.

For part~\ref{field minimal azumaya} the result is a consequence of Artin's
characterization of Azumaya algebras via identities in
\cite[Theorem~8.3]{Artin:PIAA} or \cite[Theorem~1.8.48(i/v)]{RowenPI}, since by Proposition~\ref{Rep equal
Ulr}, no homomorphic image of $\ms C$ can lie inside a matrix algebra of
degree smaller than $d$.

For parts~\ref{field azumaya center} and~\ref{field small specializations},
we construct mutually inverse morphisms $\stAzRep\phi{d} \leftrightarrow \Spec \ms Z$,
and show that the universal Azumaya algebras on the left coincides with the
Azumaya algebra $\ms C$ on the right.

Note that as in \ref{simple gerby}, since the Ulrich bundles under
consideration are line bundles (via the numerology of Proposition~\ref{Rep
equal Ulr}), the representations of degree $d$ have automorphism group
$\mbb G_m$.

Let $\Spec R$ be a $k$-algebra, and consider an object of
$\stAzRep\phi{d}(R)$ described as a representation $\ms C \to B$ where $B$
is a degree $d$ Azumaya algebra over an $k$-algebra $R$. This gives a
homomorphism of commutative $k$-algebras $\ms Z \to R$ and hence an object
of $\Spec \ms Z(R)$. Since by part~\ref{field minimal azumaya}, $\ms C$ is
Azumaya of degree $d$, it follows that $\ms C \otimes_{\ms Z} R \to B$ is a
homomorphism of Azumaya algebras over $R$ of the same rank, and hence must
be an isomorphism. Therefore the Azumaya algebra $\ms C$ on $\Spec \ms
Z$ pulls back to the canonical Azumaya algebra on $\stAzRep\phi{d}$.

In the other direction, since a homomorphism $\ms Z
\to R$ yields a representation $\ms C \to \ms C \otimes_{\ms Z} R$
which is a rank $d$ Azumaya algebra over $R$, we obtain an inverse morphism
$\Spec \ms Z \to \stAzRep\phi{d}$ as desired.
\end{proof}

\begin{cor} \label{standard azumaya}
Let $k$ be a field and $\phi: X \to Y$ a finite faithfully flat degree $d$
morphism of proper integral $k$-schemes of finite type. Then
$\CAlg\phi^{\text{red}}$ is an Azumaya algebra of degree $d$ over its
center. In particular any $k$-linear map $\CAlg\phi \to D$ for a central simple
$k$-algebra $D$ of degree $d$ must coincide with a (surjective)
specialization of the Azumaya algebra $\CAlg\phi^{\text{red}}$ with respect to
a $k$ point of $\Spec \ms Z$.
\end{cor}
\begin{proof}
Note that since $k$ is a field, the condition above ensures that $\phi$
will satisfy condition \wweird. This is then an immediate consequence of
Theorem~\ref{good azumaya love}.
\end{proof}

\section{Clifford algebras for curves} \label{curves}

In this section, we specialize to the case that $S = \Spec k$ is the
spectrum of the field, $X$ is a smooth projective $k$-curve and $Y = \mbb P^1$.
In this case, we will find that the Clifford algebra is in some sense not
sensitive to the choice of the particular morphism, but only on its degree
(see Corollary~\ref{map independence}), and that the period-index obstruction for the curve
gives some structural information about the Clifford algebra (see
Corollary~\ref{clifford decomposability}).

Let us begin with some preliminary concepts and language.
Let $X/k$ be a smooth, projective, geometrically connected curve. We recall
that the index of $X$, denoted $\ind X$ is the minimal positive degree of a
$k$-divisor on $X$. Let $\Pic X$ denote the Picard group of $X$, $\stPic X
{}$ the Picard stack of line bundles on $X$, and $\shPic X{}$ its coarse
moduli space. Write $\stPic X{n}$, $\shPic X{n}$ for the components of line
bundles of degree $n$. The Jacobian variety $J(X) = \shPic X{0}$ has the
structure of an Abelian variety under which the spaces $\shPic X{n}$ are
principal homogeneous spaces. The period of $X$, $\per X$, is the order of
$\shPic X{1}$, considered as a principal homogeneous space over the
Jacobian of $X$. The index can be considered as the minimal $n$ such that
$\shPic X{n}$ has a rational point.

We recall that we have a natural map from the $k$-rational points on
the Picard scheme of $X$ to the Brauer group of $k$, giving us an
exact sequence
\[ \Pic X \to \shPic X{}(k) \to \Br(k) \to \Br(k(X))\]
and identifying the image of the Picard scheme with the relative Brauer
group $\Br(k(X)/k)$ defined simply as the kernel of the map above on the
right (see, for example \cite[Section~3]{Clark:PI} or
\cite[Theorem~2.1]{CiKra}). In \cite[Theorem~3.5]{CiKra} it is shown that this map can be describe as being obtained from
specializing a Brauer class $\alpha_X \in \Br(\shPic X{})$, described in \cite[Lemma~3.2, Remark~3.4]{CiKra}. We define the subgroup $\Br_0(k(X)/k) \subset
\Br(k(X)/k)$ to be those elements which are images of degree $0$ classes,
i.e. $k$-points of the Jacobian of $X$. From \cite[Theorem~2.1]{CiKra}, we
have an isomorphism
\[\frac{\Br(k(X)/k)}{\Br_0(k(X)/k)} \cong \frac{\mbb Z}{(\ind X/\per
X)\mbb Z}\]

\begin{defn} \label{obstruction def}
Suppose that $\alpha \in \Br(k)$ which represents a nontrivial element of
the cyclic group $\Br(k(X)/k)/\Br_0(k(X)/k)$. Then, following
\cite{ONeil:PI} we call $\alpha$ an \textbf{(period-index) obstruction
class} for $X$.
\end{defn}

\begin{rem} \label{obstruction degree}
Let $m$ be the period of $X$. If $p \in \shPic Xm(k)$ and $q \in
\shPic X{rm}(k)$, then it follows that
$$\alpha_X|_{q} = \alpha_X|_{rp} + \alpha_X|_{q - rp} = r\alpha_X|_p +
\alpha_X|_{q - rp},$$
and so the class of $\alpha_X|_q$ is equal to the class of $\alpha_X|_{rp}$
in $\Br(X/k)/\Br_0(X/k)$, and consequently the image of any point in
$\shPic Xm(k)$ is a period-index obstruction class.
\end{rem}

\subsubsection{Relation between the universal Clifford representation space
and the universal gerbe}  \label{clifford gerbe}
Note that since the leftmost map $\Pic X \to \shPic X{}(k)$ can be
identified with the sheafification/coarse moduli map of stacks $\stPic X{}
\to \shPic X{}$ on objects defined over $k$
(a $\mathbb G_m$-gerbe, as in Remark~\ref{simple gerby}), the Brauer class
$\alpha_x$ corresponding to a $k$-point $x \in \Pic X{}(k)$ is split if and
only if the $\mbb G_m$-gerbe on $x$ obtained by pullback is split. By
a result of Amitsur (\cite[Theorem~9.3]{Am}), it follows that the Brauer class
corresponding to the gerbe $\stPic X{} \to \shPic X{}$ and the Brauer class
$\alpha_X$ defining the obstruction map generate the same cyclic subgroup
in the Brauer group of each component of $\shPic X{}$. If $X \to Y$ is a
finite morphism of degree $d$, then via the identification of the gerbe
$\stPic X{} \to \shPic X{}$ with the restriction of the gerbe
$\stRep\phi{d} \to \stAzRep\phi{d}$ of rank $d$ representations of the
Clifford algebra, which are necessarily specializations (see
Corollary~\ref{standard azumaya}) and the algebra $\CAlg\phi^{\text{red}}$,
we find that the degree $d$ specializations of the Clifford algebra
consists exactly of those central simple algebras of degree $d$ which are
obstruction classes for some Ulrich line bundle on $X$ with respect to the
$k$-morphism $\phi: X \to Y$.

\subsubsection{Stability and semistability}
For a $X$ a smooth projective curve over a field $k$ and a coherent sheaf
$V/X$, we write $\deg V = \deg c_1(V) = \int c_1(V)$ and $\mu V = \deg
V/\rk V$. Recall that a coherent sheaf is called \textbf{semistable} if for
every subsheaf $W \leq V$, we have $\mu W \leq \mu V$ and \textbf{stable} if
for every proper subsheaf $W < V$, we have $\mu W < \mu V$.

We would like to characterize in a natural way, which coherent sheaves on
$X$ will be Ulrich with respect to a finite morphism $\phi: X \to \mbb
P^1$. To do this, we have the following fact, closely following
\cite[Sections~2.1,~2.2]{VdB:Lin} and \cite{Cos}.
\begin{prop} \label{who is Ulrich}
Let $X$ be a smooth projective geometrically connected curve of genus $g$
over a field $k$. If $\phi : X \to \mbb P^1$ is a finite morphism of degree
$d$, then a coherent sheaf $V/X$ is Ulrich with respect to $\phi$ if and
only if
\begin{enumerate}
\item \label{Ulrich slope} $V$ is a semistable vector bundle on $X$ of
slope $\mu V = d + g - 1$,
\item $H^0(X, V(-1)) = 0$.
\end{enumerate}
\end{prop}
\begin{rem} \label{theta remark}
We note that the condition $H^0(X, V(-1)) = 0$ can be interpreted as
saying that the vector bundle $V(-1)$, which has slope $g - 1$, lies in the
complement of a ``generalized $\Theta$-divisor,'' (see for example
\cite{Cos, Kulk:CBi}). Recall for example, that in the
classical case, the $\Theta$-divisor is the subvariety of $\shPic X{g - 1}$
whose $\ov k$-points correspond to classes of effective divisors, and hence
those for which $H^0$ is nontrivial.
\end{rem}
\begin{proof}[Proof of Proposition~\ref{who is Ulrich}]
First, suppose that $V/X$ is Ulrich with respect to $\phi$.
Since $X$ and $\mbb P^1$ are regular of dimension $1$ and $\phi$ is flat, the sheaf $V$ is locally free if and only if
$\phi_\ast V$ is locally free. (Indeed, the torsion subsheaf would have torsion pushforward.)
The fact that $V$ must be semistable is a consequence of
\cite[Lemma~1]{VdB:Lin}. Now, supposing that $V$ is rank $r$ so that
$\phi_\ast V \cong \ms O_{\mbb P^1}^{rd}$, we use Hirzebruch-Riemann-Roch to
see
\[\chi(V) = h^0 V - h^1 V = \deg V - r(g - 1)\]
and using the fact that $\chi(V) = \chi(\phi_\ast V)$, and $h^0 \phi_\ast V = h^0
\ms O_{\mbb P^1}^{rd}  = rd$, and $h^1 \phi_\ast V = 0$, we find $\chi(V)  =
rd$. Consequently, we have
\[\deg V = rd + r(g - 1) = r(d + g - 1)\]
and so $\mu(V) = d + g - 1$, as claimed.
For the other condition, we note that
\[H^0(X, V(-1)) = H^0(\mbb P^1, \phi_\ast V(-1)) = H^0(\mbb P^1, \ms
O(-1)^{rd}) = 0. \]

For the converse, let us assume that $V/X$ has slope $d + g - 1$ and
$H^0(X, V(-1)) = 0$. By the result of
Birkhoff-Grothendieck-Hazewinkel-Martin \cite[Theorem~4.1]{HazMar}, we can
write
\[\phi_\ast V \cong \oplus \ms O(n_i)\]
for some collection of integers $n_i$. We claim that all the $n_i$ are
equal to $0$, which would imply the result.

The condition that $H^0(X, V(-1)) = 0$ tells us that all the $n_i$ are
nonpositive. It therefore follows that we have $h^0(V)$ is precisely the
number of indices $i$ such that $n_i$ is equal to $0$. It follows from
Hirzebruch-Riemann-Roch that
\[\chi(V) = \deg - r(g-1) = rd + r(g - 1) - r(g
- 1) = rd,\]
and therefore $h^0 V = rd + h^1 V$, which implies that $h^0 V \geq rd$. But
this implies that at least $rd$ of the integers $n_i$ are nonnegative, and
therefore all are $0$, as claimed.
\end{proof}

\begin{cor} \label{map independence}
Suppose that $\phi, \phi' : X \to \mbb P^1$ are two degree $d$ morphisms of
curves. Then there exists a line bundle $\ms N$ on $X$ such that
tensoring by $\ms N$ gives a equivalence between the Ulrich bundles with
respect to $\phi$ and the Ulrich bundles with respect to $\phi'$.
\end{cor}
\begin{proof}
Let $\ms L, \ms L'$ be the pullbacks of $\mbb O_{\mbb P^1}(-1)$ under
$\phi$ and $\phi'$ respectively. Since these are both degree $d$ line
bundles, we can write $\ms L \otimes \ms N = \ms L'$ for some line bundle
$\ms N$ of degree $0$. By Proposition~\ref{who is Ulrich}, we then find that a coherent
sheaf $V/X$ is Ulrich with respect to $\phi$ if and only if $H^0(X, V
\otimes \ms L^\vee) = 0$. But
\[V \otimes \ms L^\vee = V \otimes (\ms L' \otimes \ms N^\vee)^\vee = (\ms N \otimes V) \otimes \ms L'^\vee\]
and so $H^0(X, V
\otimes \ms L^\vee) =
H^0(X, (\ms N \otimes V)
\otimes \ms L'^\vee)$
and we see that $V/X$ is Ulrich with respect to $\phi$ if and only if
$\ms N \otimes V$ is Ulrich with respect to $\phi'$.
\end{proof}
It follows that the the stack of Ulrich bundles is independent of
the specific morphism $\phi$, and only depends on its degree $d$. In
particular, by Theorem~\ref{good azumaya love}(\ref{field small
specializations}), the center of the reduced Clifford algebra
$\CAlg\phi^{\text{red}}$ and its Brauer class over its center only depend
on $d$ and not on the specific choice of $\phi$.

The relative Brauer map and related period-index obstruction
(Definition~\ref{obstruction def}) have been the
subject of a great deal of arithmetic investigations
(see, for example \cite{LangTate,Sha,CiKra,Clark:PI,Lich:PI,Roq:Curves}). An
interesting aspect of the study of the Clifford algebra is that it gives
another concrete interpretation of this morphism, and understanding
its specializations can yield nontrivial arithmetic
information about a curve.

In this direction we give a result on the specializations of the Clifford
algebra of a curve. In \cite{HaileHan} and \cite{Haile:CABC}, it is shown that
in certain cases of a Clifford algebra associated to a genus $1$
hyperelliptic or plane cubic curve, the Clifford algebra specializes to any
division algebra of degree $2$ or $3$ respectively, which is split by the
function field of the genus $1$ curve. The following result gives a natural
generalization of this result for general curves.
\begin{prop} \label{special division}
Let $X/k$ be a geometrically integral smooth projective curve of index $d$
over $k$, which admits a degree $d$ morphism $\phi: X \to \mbb P^1$, and
suppose $D$ is a division algebra of degree $d$ over its center $k$, such that $D_{k(X)}$ is
split. Then there is a specialization $\CAlg\phi \to D$.
\end{prop}
\begin{rem} \label{special remark}
It follows from \ref{who is Ulrich} that such a class $[D]$ must arise
as the obstruction class for the gerbe $\stPic X{d + g - 1} \to \shPic
X{d + g - 1}$.
\end{rem}
\begin{proof}[Proof of Proposition~\ref{special division}]
Let $\ms G\to\spec k$ be a $\m_d$-gerbe representing the Brauer
class of $D$, and write $D=\send(W)$ for a $\ms G$-twisted sheaf $W$
of rank $d$.  Let $\ms P\to\bP^1$ and $\ms X\to X$ be the pullbacks
of $\ms G$ to $\bP^1$ and $X$.  Let $\ms L$ be a $\ms X$-twisted
invertible sheaf, and let $V:=\phi_\ast\ms L$ be the pushforward
$\ms P$-twisted sheaf.

We claim that there is an integer $n$ and an isomorphism $V(n) \simto
W_{\ms P}$.  Note that $V$ is naturally a $\phi_\ast\ms O_{\ms X}$-module,
giving a map $\phi_\ast\ms O_{\ms X}\to\send(W)$.  The claim thus yields a
map $\phi_\ast\ms O_{\ms X}\to\send(W_{\ms P})=D_{\ms P}$.  Since both
sheaves have trivial inertial action, this is the pullback of a unique map
$\phi_{\ast}\ms O_X\to D_{\bP^1}$, which by Corollary~\ref{standard
azumaya}, comes from a specialization of the Clifford algebra $\ms
C_{X/\mbb P^1} \to D$.

It remains to prove the claim. For this, note that $V$ is a $\ms
P$-twisted sheaf of minimal rank (as the index of $\ms P$ is equal to $d$,
which is the rank of $V$). In particular, it follows that $V$ must be
stable of some slope $\mu$, since otherwise the Jordan-H\"older filtration
will yield a twisted sheaf of smaller
rank.  Further, it follows by the same reasoning that $V_{\ov k}$ must be
equal to its $\mu$-socle, the sum of its $\mu$-stable subsheaves. Hence,
$V$ must be geometricaly polystable, which, implies that $W\tensor\widebar
k=\ms L(m)^d$ for some fixed $m$ and an invertible $\ms P\tensor\widebar
k$-twisted sheaf $\ms L$ of degree $0$.

On the other hand, $W_{\ms P}$ is also a locally free $\ms P$-twisted sheaf
of rank $d$, hence also geometrically polystable.  It follows that there is
an integer $n$ such that $V(n)\tensor\widebar k$ and $W_{\ms
P}\tensor\widebar k$ are isomorphic over $\ms P\tensor\widebar k$.  The
space $I:=\isom(V(n),W_{\ms P})$ is thus a right $\aut(W)$-torsor which is
open in the (positive-dimensional) affine space $\Hom(V(n),W)$.  If $k$ is
infinite, it follows that $I$ has a rational point (as the rational points
are dense in any open subset of an affine space); if $k$ is finite, then
$I$ has a rational point because any torsor under a smooth connected
$k$-group scheme of finite type is split by Lang's theorem
\cite[Theorem~2]{Lang:AGFF}. In either case, we see that $V(n)$ and $W_{\ms
P}$ are isomorphic, verifying the claim.
\end{proof}

\subsection{Genus $1$ curves}

For the remainder, we will focus on the case where $X$ is a curve of genus
$1$ over $k$.  Let us begin with the following theorem, which illustrated
the connection between our Clifford algebras and the arithmetic of genus
$1$ curves:

\begin{thm} \label{genus 1 observations}
Suppose that $X/k$ is a genus $1$ curve of index $d > 1$. Then
\begin{enumerate}[1. ]
\item \label{genus 1 map}
$X$ admits a
degree $d$ finite morphism $\phi : X \to \mbb P^1$.
\item \label{ulrich genus 1 locus}
For such a morphism, the Ulrich locus of $\stPic X{}$ lies within the
component $\stPic Xd$.
\item \label{degree 0 equiv}
We have an equivalence of stacks $\stPic Xd \cong \stPic X0$.
\item \label{obstruction not special}
No specialization of the Clifford algebra is a period-index obstruction
algebra for $X$ (as in Definition~\ref{obstruction def}).
\end{enumerate}
\end{thm}
\begin{proof}
For part \ref{genus 1 map},
we note that for such a curve, if we choose a divisor $D$ of degree $d$,
then by Riemann-Roch, $h^0(D) = \deg D \geq 2$, and so we see that we can
find a pencil of effective divisors providing a degree $d$ morphism $\phi :
X \to \mbb P^1$.

Part~\ref{ulrich genus 1 locus} follows from Proposition~\ref{who is
Ulrich}(\ref{Ulrich slope}).

For part~\ref{degree 0 equiv} observe that, since $X$ has a effective
divisor of degree $d$ defined over $k$, by adding and subtracting that
point, we obtain an equivalence of stacks $\stPic Xd \cong \stPic X0$.

Finally, part~\ref{obstruction not special} is now a consequence of
Remark~\ref{special remark}.
\end{proof}

\subsection{Decomposability of the period-index obstruction}

We consider now the concept of decomposability of algebras. We recall that
a central simple algebra $A$ is decomposable if it can be written as $A
\cong B \otimes C$ for two nontrivial algebras $B, C$. Circumstantial
evidence would seem to suggest that algebras which have index ``maximally
different'' from the period should be decomposable. Results of this type
have been obtained by Suresh (\cite[Theorem~2.4]{Suresh:SLGC}, see also
\cite[Remark~4.5]{BMT}),  over fields such as $\mbb Q_p(t)$.  We begin by
giving the following weakening of the standard notion of decomposability:
\begin{defn}
We say that a central simple algebra $A$ is weakly decomposable if
there exist central simple algebras $B, C$ of degree greater than $1$, such
that $\deg B, \deg C$ divide but are strictly less than $\deg A$ and $A$ is
Brauer equivalent to $B \otimes C$.
\end{defn}

\begin{rem}
We note that in the case that $A$ is a weakly decomposable algebra of
degree $p^2$, then $A$ is in fact decomposable.
\end{rem}

\begin{prop} \label{weak decomp}
Let $X/k$ be a smooth genus $1$ curve. Then every class in the relative Brauer group $\Br(k(X)/k)$ can be
written as $[B \otimes D]$ where $\ind B, \ind D | per X$. In particular,
if $\ind X \neq \per X$, and $A$ is a division algebra of index $d$ whose
class is in $\Br(k(X)/k)$, then $A$ is weakly decomposable.
\end{prop}
\begin{proof}
If $A$ is a central simple algebra of degree $d$ which is not division then
the result is immediate. Therefore we may assume that $A$ is a central
division algebra of degree $d$, and let $D$ be any obstruction class for
$X$. By Proposition~\ref{special division}, $A$ is a specialization of the
Clifford algebra, and hence by Theorem~\ref{genus 1 observations} its class
must lie in $\Br_0(X/k)$ (see the discussion at the beginning of
Section~\ref{curves}), and therefore $A \otimes D^{op}$ must be a
period-index obstruction class.  By \cite[Proposition~2.3]{ONeil:PI}, we
find that $\ind D, \ind A \otimes D^{op} | \per X < d$. Setting $B$ to be a
division algebra Brauer equivalent to $A \otimes D^{op}$, we find $A \sim D
\otimes B$, showing that $A$ is weakly decomposable as desired.
\end{proof}

Using this, we find that the Clifford algebra is also weakly decomposable
in this situation.
\begin{cor} \label{clifford decomposability}
Let $X/k$ be a smooth genus $1$ curve, and let $\phi: X \to \mbb P^1$ be a
morphism of degree $\ind X$. Let $C$ be the specialization of the reduced
Clifford algebra $\ms C_\phi^{\text{red}}$ to the generic point of $\shPic
X{\ind X}$. If $\per X \neq \ind X$ then $C$ is weakly decomposable.
\end{cor}
\begin{proof}
Let $L = k(\shPic X{\ind X})$ be the function field of $\shPic X{\ind X}$,
and let $\eta \in \shPic X{\ind X}(L)$ be the generic point. Regarding $C$
as the specialization of $\ms C_{\phi_L}^{\text{red}}$ to $\eta \in \shPic
X{\ind X}(L) = \shPic {X_L}{\ind X}(L)$, it follows from
Section~\ref{clifford gerbe} that $C \in \Br(X_L/L)$, and the result now
follows from Proposition~\ref{weak decomp}.
\end{proof}

\appendix

\section{Explicit constructions} \label{examples}

In this appendix, we relate our Clifford algebra functor (and therefore our
constructed Clifford algebras) to the more classical constructions in the
literature and their natural generalizations.  With this in mind, we
present Clifford algebras in a number of generalizations of previously seen
contexts, each time showing how previous constructions fit within this
description.  It turns out that all the existing descriptions of Clifford
algebras can be all seen as particular examples of the Clifford algebra
associated to projection of a hypersurfaces in certain weighted projective
spaces. We finish by giving an explicit version of the general existence
proof for such Clifford algebras, giving an explicit presentation for such
Clifford algebras.

To set notational conventions, we will assume that all rings and algebras are
associative and unital and their homomorphism are unital.  For a ring $R$,
we let $R\langle x_1, \ldots, x_n \rangle$ denote the free associative
algebra over $R$ generated by the $x_i$.

\subsection{The Clifford algebra of a homogeneous polynomial}
Recall, if $f$ is a degree $d$ homogeneous polynomial in the variables
$x_1, \ldots, x_n$, following \cite{Roby}, we define the Clifford algebra
of $f$, denoted $C(f)$ by
\[C(f) = k \langle a_1, \ldots, a_n \rangle/I\]

where $I$ is the ideal generated by the coefficients of the variables
$x_i$ in the expression
\[(a_1 x_1 + \cdots + a_n x_n)^d - f(x_1, \ldots, x_n) \in k\langle
a_1, \ldots, a_n \rangle [x_1, \ldots, x_n].\]

\begin{prop}
Suppose $f$ is a degree $d$ homogeneous polynomial in the variables
$x_1, \ldots, x_n$. Let $X$ be the hypersurface defined by the
equation $x_0 ^d - f(x_1, \ldots, x_n)$, and let $\phi : X \to
\bP^{n-1}$ be the degree $d$ morphism given by dropping the $x_0$-coordinate.
Then $\CFun\phi$ is represented by the algbra $C(f)$.
\end{prop}
\begin{proof}
This follows from Theorem~\ref{explicit general clifford}
\end{proof}

\subsection{Weighted Clifford algebras of homogeneous polynomials}
This construction is a generalization of the hyperelliptic Clifford
algebras introduced by Haile and Han in \cite{HaileHan}.

For positive integers $m, d$, let $f$ is a degree $md$ homogeneous
polynomial in the variables $x_1, \ldots, x_n$. We define the Clifford
algebra of $f$, weighted by $m$ denoted $C_m(f)$ by
\[C_m(f) = k \langle a_1, \ldots, a_n \rangle/I\]
where $I$ is the ideal generated by the coefficients of the variables
$x_i$ in the expression
\[\left(\sum_{|J| = m} a_J x^J\right)^d - f(x_1, \ldots, x_n) \in k\langle
a_J \rangle_{|J| = m} [x_1, \ldots, x_n],\]
where $x^J = x_1^{j_1}\cdots x_n^{j_n}$ ranges through all monomials
of degree $m$.

\begin{prop}
Suppose $f$ is a degree $md$ homogeneous polynomial in the variables
$x_1, \ldots, x_n$, and let $d$ be a positive integer. Consider $\mbb
P = \mbb P_{m, 1, \ldots, 1}$, a weighted
projective $n$ space. Let $X$ be the hypersurface defined by the
equation $x_0 ^d - f(x_1, \ldots, x_n)$, and let $\phi : X \to
\bP^{n-1}$ be the degree $d$ morphism given by dropping the $x_0$-coordinate.
Then $\ms C(\phi)$ is represented by the algbra $C_m(f)$.
\end{prop}
\begin{proof}
This follows from Theorem~\ref{explicit general clifford}
\end{proof}

\subsection{Non-diagonal Clifford algebras of homogeneous polynomials}
This version of a Clifford algebra construction is due to Pappacena
\cite{Papp}. Particularly interesting case are the non-diagonal Clifford
algebras of a binary cubic form, studied by Kuo in \cite{Kuo}, and in
somewhat more generality by Chapman and Kuo in \cite{ChaKuo}.

For a positive integer $d$, suppose that we are given $f_1, f_2,
\ldots, f_d \in k[x_1, \ldots, x_n]$ where $f_i$ is homogeneous of
degree $i$. We define $C(f_1, \ldots, f_n)$ to be the associative
$k$-algebra given by
\[C(f_1, \ldots, f_n) = k \langle a_1, \ldots, a_n \rangle/I\]
where $I$ is the ideal generated by the coefficients of the variables
$x_i$ in the expression
\begin{multline*}
(a_1 x_1 + \cdots + a_n x_n)^d - (a_1 x_1 + \cdots + a_n x_n)^{d-1}
f_1(x_1, \ldots, x_n) \\ - (a_1 x_1 + \cdots + a_n x_n)^{d-2}
f_2(x_1, \ldots, x_n) - \cdots - f_d(x_1, \ldots, x_n) \in k\langle
a_1, \ldots, a_n \rangle [x_1, \ldots, x_n].
\end{multline*}

\begin{prop}
Suppose we are given polynomials $f_1, \ldots, f_n$ in the variables
$x_1, \ldots, x_n$, where $f_i$ is homogeneous of degree $i$. Let $X$
be the hypersurface in $\mbb P^n$ defined by the equation
\[ x_0^d = x_0^{d-1} f_1 + x_0^{d-2} f_2 + \cdots + f_d,\]
and let $\phi : X \to \bP^{n-1}$ be the degree $d$ morphism given by
dropping the $x_0$-coordinate.  Then $\CFun\phi$ is represented by
the algbra $C(f_1, \ldots, f_d)$. In particular, $C(f_1, \ldots, f_d) \cong
\CAlg\phi$.
\end{prop}
\begin{proof}
This follows from Theorem~\ref{explicit general clifford}
\end{proof}

\subsection{Weighted non-diagonal Clifford algebras of homogeneous
polynomials}

This construction is a common generalization of the previous
constructions. Suppose that we are given
$f_m, f_{2m}, \ldots, f_{dm} \in k[x_1, \ldots, x_n]$ where $f_i$ is
homogeneous of degree $i$. We define $C(f_m, \ldots, f_{dm})$ to be the
associative $k$-algebra given by
\[C(f_{m}, \ldots, f_{dm}) = k \langle a_J\rangle_{|J| = m}/I\]
where $I$ is the ideal generated by the coefficients of the variables
$x_i$ in the expression
\begin{equation} \label{clifford equation}
\left(\sum_{|J| = m} a_J x^J\right)^d =
\sum_{\ell = 1}^m \left(\left(\sum_{|J| = m} a_J x^J\right)^{d-\ell} f_{\ell
m}(x_1, \ldots, x_n) \right)\in k\langle a_J\rangle_{|J|=m} [x_1, \ldots, x_n].
\end{equation}

\begin{thm} \label{explicit general clifford}
Suppose that we are given $f_m, f_{2m}, \ldots, f_{dm} \in k[x_1,
\ldots, x_n]$ where $f_i$ is homogeneous of degree $i$. Let $X$ be the
hypersurface in the weighted projective space $\mbb P = \mbb P_{m, 1,
\ldots, 1}$ defined by the degree $md$ homogeneous equation
\[ x_0^d = x_0^{d-1} f_m + x_0^{d-2} f_{2m} + \cdots + f_{dm},\]
and let $\phi : X \to \bP^{n-1}$ be the degree $d$ morphism given by
dropping the $x_0$-coordinate.  Then $\CFun\phi$ is represented by
the algbra $C(f_m, \ldots, f_{md})$.
\end{thm}

\begin{proof}
To begin, let us examine the morphism $\phi$ in local coordinates. Let
$U_i \cong \mbb A^{n-1}$ be the affine open set of $\mbb P^{n-1}$
defined by the nonvanishing of the coordinate $x_i$, so that $U_i =
\Spec(R_i)$ where $R_i = k[x_1/x_i, \ldots, x_n/x_i] \subset k(\mbb
P^{n-1})$.  Similarly, let $V_i \subset \mbb P$ be defined by the
nonvanishing of the $x_i$ coordinate on $\mbb P$. If we write $V_i =
\Spec(S_i)$ then we have
\[S_0 = k[x^J/x_0]_{|J| = m}\]
where $x^J = x_1^{j_1} \cdots x_n^{j_n}$ is a monomial of degree $|J|
= m$. For $i \neq 0$, we have
\[S_i = k[x_0/x_i^m, x_1/x_i, \ldots, x_n/x_i]\]
which we note is just a polynomial ring in one variable, represented
by $x_0/x_i^m$ over the ring $R_i$. If we let $X_i = X \cap V_i$, then
we see that (via homogenization) $X_0$ is cut out by the equation
\[ 1 =  f_m/x_0 + f_{2m}/x_0^2 + \cdots + f_{dm}/x_0^d.\]
We claim that in fact $X_0 \subset \cup_{i = 1}^n X_i$, or in other
words, $X$ is contained in the union of the open sets $V_1, \ldots,
V_n$. To see this, suppose that $p \in X_0(L)$ is a point for some
field extension $L/F$. It follows that for some $J$ with $|J|=m$, we
have $x^J/x_0(p) \neq 0$, since otherwise we would have $f_{\ell
m}/x_0^\ell (p) = 0$ for each $\ell$ contradicting the equation above.
Now, we can choose $i$ with $j_i \neq 0$ --- i.e., so that $x_i$
appears with a nonzero multiplicity in the monomial $x^J$. We claim
that $p \in X_i$. But this follows by construction:  $p$ does not
lie in the zero set of the homogeneous polynomial $x_i$, the ideal of
which on the affine set $V_i$ contains the term $x^J/x_0^m$.

For $i \neq 0$, in the affine set $V_i = \Spec(S_i) =
\Spec(k[x_0/x_i^m, x_1/x_i, \ldots, x_n/x_i])$, the closed subscheme
$X_i$ is cut out by the equation
\[ (x_0/x_i^m)^d = (x_0/x_i^m)^{d-1} f_m(x_1/x_i, \ldots, x_n/x_i) +
\cdots + f_{dm}(x_1/x_i, \ldots, x_n/x_i).\]

Let $A = k \langle a_1, \ldots, a_n \rangle/I$, where the ideal $I$ is
as described above. We will show that $A$ represents the functor $\ms
C(\phi)$. For the first direction, note that we have a homomorphisms
of sheaves of $\ms O_{\mbb P^{n-1}}$-algebras $\phi_\ast \ms O_{X} \to A
\otimes_k \ms O_{\mbb P^{n-1}}$ given on the open set $U_i$ by the
inclusion
\[S_i \to A \otimes_k R_i\]
defined by sending $x_0/x_i^m$ to $\sum_{|J| = m} a_J \otimes x^J$. It
follows from multiplying the defining equation~(\ref{clifford
equation}) by $x_i^{-md}$ that this defines a homomorphism of
$R_i$-algebras.

Conversely, suppose that $B$ is any $k$-algebra, and that we have a
homomorphism of sheaves of $\ms O_{\mbb P^{n-1}}$-algebras $\phi_\ast \ms
O_{X} \to B \otimes_k \ms O_{\mbb P^{n-1}}$. Over the
open set $U_i$ write $b_i$ for the image of $x_0/x_i^m \in S_i$ in $B
\otimes_k R_i$, we see that since $b_i x_i^m/x_j^m = b_j \in B \otimes
R_j$, it follows that $b_i$, considered as a polynomial in $x_1/x_i,
\ldots, x_n/x_i$ can have degree no larger than $m$. In particular, we
may write
\[b_i x_i^m = \sum_{|J| = m} \beta_{J,i} x^J\]
and using the identity $b_i x_i^m = b_j x_j^m$, it follows that the
elements $\beta_{J,i} = \beta_J \in B$ do not in fact depend on $i$.
But now, $a_J \mapsto \beta_J$ defines a homomorphism $A \to B$
such that $S_i \to B \otimes R_i$ factors as
\[S_i \to A \otimes R_i \to B \otimes R_i\]
as desired.

\end{proof}

\nocite{Haile:CABC,Haile:CACSFF,Haile:CABCS}
\bibliographystyle{amsalpha}
\def\cprime{$'$} \def\cprime{$'$} \def\cprime{$'$} \def\cprime{$'$}
  \def\cftil#1{\ifmmode\setbox7\hbox{$\accent"5E#1$}\else
  \setbox7\hbox{\accent"5E#1}\penalty 10000\relax\fi\raise 1\ht7
  \hbox{\lower1.15ex\hbox to 1\wd7{\hss\accent"7E\hss}}\penalty 10000
  \hskip-1\wd7\penalty 10000\box7}
\providecommand{\bysame}{\leavevmode\hbox to3em{\hrulefill}\thinspace}
\providecommand{\MR}{\relax\ifhmode\unskip\space\fi MR }
\providecommand{\MRhref}[2]{%
  \href{http://www.ams.org/mathscinet-getitem?mr=#1}{#2}
}
\providecommand{\href}[2]{#2}

\end{document}